\documentclass[a4paper]{article}
\usepackage{amsmath}
\usepackage{amsfonts,amssymb,amsthm}
\usepackage[final]{graphicx}

\newtheorem{thm}{Theorem}
\newtheorem{corr}{Corollary}
\newtheorem{lem}{Lemma}
\newtheorem{prop}{Proposition}
\newtheorem{rem}{Remark}
\newtheorem{df}{Definition}

\newcommand{\RM}{\mathbb{R}}

\newcommand{\CM}{\mathbb{C}}
\newcommand{\MM}{\,\mbox{\bf M}}

\newcommand{\HM}{\,\mbox{\bf H}}

\newcommand{\mm}{\;\mbox{\bf m}}

\newcommand{\tr}{\,\mbox{\rm tr}}

\newcommand{\cn}{\operatorname{cn}}

\newcommand{\spec}{\operatorname{spec}}

\title{Nonlinear Stability of Periodic Traveling Wave Solutions of the Generalized Korteweg-de Vries Equation}
\author{Mathew A. Johnson}

\begin{document}
\bibliographystyle{plain}
\maketitle

\begin{abstract}
In this paper, we study the orbital stability for a four-parameter family of periodic stationary traveling wave
solutions to the generalized Korteweg-de Vries (gKdV) equation
\[
u_t=u_{xxx}+f(u)_x.
\]
In particular, we derive sufficient conditions for such a solution to be orbitally stable in terms
of the Hessian of the classical action of the corresponding traveling wave ordinary differential
equation restricted to the manifold of periodic traveling wave solution.
We show this condition is equivalent to the solution being spectrally stable with respect
to perturbations of the same period in the case when $f(u)=u^2$ (the Korteweg-de Vries equation)
and in neighborhoods of the homoclinic and equilibrium solutions if $f(u)=u^{p+1}$ for some $p\geq 1$.
\end{abstract}


\section{Introduction}
This paper concerns the stability analysis of periodic traveling wave solutions of the generalized
Korteweg-de Vries (gKdV) equation
\begin{equation}
u_t=u_{xxx}+f(u)_x\label{gkdv}
\end{equation}
where $f$ is a sufficiently smooth non-linearity satisfying certain convexity assumptions.
Probably the most famous equation among this family is given by $f(u)=u^2$, in which case
\eqref{gkdv} corresponds to the Korteweg-de Vries (KdV) equation.  The KdV
serves as an approximate description of small amplitude waves propagating in a
weakly dispersive media.  When the amplitude of the wave is not small, however, the gKdV
\eqref{gkdv} arises where the nonlinearity $f$ describes the effects of the solution
interacting with itself through the nonlinear media.

It is well known that the gKdV equation admits traveling wave solutions of the form
\begin{equation}
u(x,t)=u_c(x+ct),\;\;x\in\RM,\;\;t\in\RM\label{travel}
\end{equation}
for wave speeds $c>0$.  Historically, there has been much interest in the stability of traveling solitary waves of the form \eqref{travel}
where the profile $u_c$ decays exponentially to zero as its argument becomes unbounded.  Such waves
were initially discovered by Scott Russell in the case of the KdV where the traveling
wave is termed a soliton.  While \eqref{gkdv} does not in general posess exact ``soliton" solutions,
which requires complete integrability of the partial differential equation, 
exponentially decaying traveling wave solutions still exist.  Moreover, the stability of
such solitary waves is  well understood and dates back to the pioneering work of Benjamin \cite{Ben},
which was then further developed by Bona \cite{Bona}, Grillakis \cite{G}, Grillakis, Shatah and
Strauss \cite{GSS}, Bona, Souganides and Strauss \cite{BSS},
Pego and Weinstein \cite{PW2,PW1}, Weinstein\cite{MIW1,MIW2}, and many others.
In this theory, it is shown that traveling solitary waves of \eqref{gkdv} are orbitally stable if the the solitary
wave stability index
\begin{equation}
\frac{\partial}{\partial c}\int_{-\infty}^\infty u_c^2\;dx\label{solitaryindex}
\end{equation}
is positive, and is orbitally unstable if this index is negative.  In the case where \eqref{gkdv} has a
power-law nonlinearity $f(u)=u^{p+1}$, the sign of this stability index is positive
if $p<4$ and is negative if $p>4$.  Moreover, in \cite{PW2,PW1} it was shown that the mechanism
for this instability is as follows: Linearizing the traveling wave partial differential equation
\begin{equation}
u_t=u_{xxx}+f(u)_x-cu_x,\label{travelgkdv}
\end{equation}
which is satisfied by the traveling solitary wave profile, about the solution $u_c$
and taking the Fourier transform in time leads to a spectral problem of the form
\[
\partial_x\mathcal{L}[u_c]v=\mu v
\]
considered on the real Hilbert space $L^2(\RM)$, where $\mathcal{L}[u_c]$ is a second
order self adjoint differential operator with asymptotically constant
coefficients.  The authors then make a detailed study of the Evans function $D(\mu)$, which is an analytic function
such that if if $\psi$ is a solution of \eqref{travelgkdv}
satisfying $\psi(x)\sim e^{\omega x}\;\;{\rm as }\;x\to\infty$, then
$\psi(x)\sim D(\mu)e^{\omega x}\;\;{\rm as }\;x\to-\infty$: in essence, $D(\mu)$ plays the role
of a transmission coefficient familiar from quantum scattering theory.
This approach motivated by the fact that for ${\rm Re}(\mu)>0$ the vanishing of $D(\mu)$
implies that $\mu$ is an $L^2$ eigenvalue of the
linearized operator $\partial_x\mathcal{L}[u_c]$, and conversely.  Pego and Weinstein were able
to use this machinery to prove that the Evans function satisfies
\[
\lim_{\mu\to+\infty}{\rm sign}(D(\mu))>0
\]
as well as the asymptotic relation
\[
D(\mu)=C_1\left(\frac{\partial}{\partial c}\int_{-\infty}^\infty u_c(x)^2dx\right)\mu^2+o(|\mu|^2)
\]
in a neighborhood of $\mu=0$, for some positive constant $C_1$.  Thus, in the case when the solitary wave stability index is negative,
it follows by the continuity of $D(\mu)$ for $\mu\in\RM^+$ that $D(\mu)<0$
for small positive $\mu$ and hence $D(\mu)$ must have a positive root, thus proving
exponential instability of the underlying traveling solitary wave in this case.  

In this paper, however, we are concerned with traveling wave solutions of \eqref{gkdv} of
the form \eqref{travel}, where this time we require the profile $u_c$ be a periodic
function of its argument.  In contrast to the traveling solitary wave theory, 
relatively little is known concerning the stability of periodic traveling waves of
nonlinear dispersive equations such as the gKdV.  Existing results usually come two types: spectral stability
with respect to localized or bounded perturbations, and orbital (nonlinear) stability with
respect to periodic perturbations.  Most spectral stability results seem to rely on a Floquet-Bloch decomposition
of the linearized operator and a detailed analysis of the resulting family of spectral problems, or else perturbation
techniques which analyze modulational instability (spectrum near the origin).  Many of the 
stability results known to the author for the traveling wave solutions of the gKdV
concern the stability of the cnoidal solutions of the KdV
\[
u(x,t)=u_0+12k^2\kappa^2\cn^2\left(\kappa\left(x-x_0+\left(8k^2\kappa^2-4\kappa^2+u_0\right)t\right),k\right),
\]
where $k\in[0,1)$ and $\kappa,\;x_0,$ and $u_0$ are real constants.  Such cnoidal solutions of the KdV
have been studied by McKean \cite{McKean}, and more recently in papers by Pava, Bona, and Scialom \cite{PBS}
and by Bottman and Deconinck \cite{BD}.  The results in \cite{McKean} uses the complete integrability
of the KdV in his study of the periodic initial value problem in order to show nonlinear stability of
the cnoidal solutions to perturbations of the same period.  Also using the machinery of complete
integrability, in \cite{BD} the spectrum of the linearized operator on the Hilbert space $L^2(\RM)$
is explicitly computed and shown to be confined to the imaginary axis.
In particular, it follows that cnoidal solutions of the KdV are spectrally stable to perturbations of the same period,
and more generally, perturbations with periods which are integer multiples of the period of the cnoidal wave.

In \cite{PBS}, more classical energy functional methods are used to show that such cnoidal
solutions are orbitally stable to perturbations of the same period as the cnoidal wave.  Since their
techniques do not rely on the complete integrability of the KdV, it is possible that their
methods can be extended to equations in the family described by \eqref{gkdv} other than the KdV.
In this paper, conduct exactly such an analysis and show how the orbital stability of the cnoidal solutions
of the KdV can be recovered from this theory once it is known that such solutions are spectrally
stable to perturbations of the same period.  It should be noted, however, that our analysis varies
slightly from that of \cite{PBS} due to the fact that the \eqref{gkdv} is not in general Galilean invariant.
Thus, we must make use of all three conserved quantities of the gKdV flow: see Remark \ref{kdv:rem} below.
Our goal is to derive geometric criterion for the orbital stability of such solutions which
are universal to the family of equations \eqref{gkdv}.  As such, our results are not as
explicit as in the papers mentioned previously, yet they apply to a much broader range of equations.

Returning to the generalized KdV equation \eqref{gkdv}, spectral stability results have recently been
obtained by Haragus and Kapitula \cite{HK} as well as by the author in collaboration with Bronski \cite{BrJ}.
In \cite{HK}, the spectral stability of small amplitude periodic traveling wave solutions of \eqref{gkdv}
with $f(u)=u^{p+1}$ was studied.  By using a Floquet-Bloch decomposition of the linearized spectral problem,
the authors found that such solutions\footnote{In the next section we will give a parametrization of the
periodic traveling wave solutions of \eqref{gkdv}.  In \cite{HK}, they only consider small amplitude
periodic traveling waves with the integration parameter $a$ (see section 2 for definition) being small.}
are spectrally stable if $p\in[1,2)$ and exhibit a
modulational instability if $p>2$.  In particular, they found that such solutions are always spectrally
stable to perturbations of the same period: in section 5, we will verify and extend this result
through the use of the periodic Evans function.  
In \cite{BrJ}, a modulational instability analysis of
the periodic traveling wave solutions of \eqref{gkdv} by using Floquet theory and developing a perturbation
theory appropriate to the Jordan structure of the period map at the origin.
As a by product of their analysis, a sufficient condition for
exponential instability of the underlying periodic traveling wave with respect to periodic
perturbations was obtained in terms of the conserved quantities of the gKdV flow.
In particular, a stability index was derived 
in a manner quite similar to the solitary wave theory outlined above such that the negativity
of this index implies exponential instability of the periodic traveling wave with respect to perturbations
of the same period.  The relevant results of this analysis can be found in section $3$.

It seems natural to wonder role this periodic instability index derived in \cite{BrJ}
plays in the nonlinear stability of the periodic
traveling wave.  As mentioned above, the analogue of this index controls the nonlinear stability in
the solitary wave context.  Thus, one would like the periodic traveling wave to be nonlinearly stable
whenever the aforementioned periodic stability index positive.  While we are able to show this
is true in certain cases, we find that two other quantities, which are essentially not present in the
spectral stability theory\footnote{This is not quite correct.  They are present, but their signs
do not play into the spectral stability theory.  See section 3 for more details.} nor the solitary
wave theory, play a role in the nonlinear  stability.   This is the content of our main theorem, which is stated below
\footnote{For the definition of the real Hilbert space $X$, see section $4$.}.

\begin{thm}\label{os:state}
Let $u(x+ct)$ be a periodic traveling wave solution of \eqref{gkdv}, whose orbit is bounded in
phase space by the homoclinic orbit.  Moreover, assume the periodic instability index $\{T,M,P\}_{a,E,c}$
is positive.  Then there exist two real real numbers, denoted $T_E$ and $\{T,M\}_{a,E}$,
associated to the periodic traveling wave profile $u$ such that if $T_E> 0$ and $\{T,M\}_{a,E}>0$, then
there exists $C_0>0$ and $\varepsilon>0$ such that
for all $\phi_0\in X$ with $\|\phi_0\|_{X}<\varepsilon$, the solution $\phi(x,t)$ of \eqref{gkdv}
with initial data $\phi(x,0)=u(x)+\phi(x)$ satisfies
\[
\inf_{\xi\in\RM}\|\phi(\cdot,t)-u(x+ct+\xi)\|_{X}\leq C_0\|\phi_0\|_{X}
\]
for all $t>0$.
\end{thm}

\begin{rem}
Throughout this paper, ``orbital stability" will always mean orbital stability with respect
to perturbations of the same period as the underlying wave.
\end{rem}

The outline for this paper is as follows.  Section $2$ will be devoted to a study of the basic
properties of the periodic traveling wave solution of equation \eqref{gkdv}.  In section $3$, we will
recall the recent results of \cite{BrJ} concerning the spectral stability of periodic traveling wave solutions
of \eqref{gkdv} with respect to perturbations of the same period.  The resulting instability
index will play an important role throughout the rest of the paper.  Section $4$ is devoted
to the proof of Theorem \ref{os:state}.  Finally, two applications of our theory are described
in sections 5 and 6 in the case of a power-law nonlinearity $f(u)=u^{p+1}$ for $p\geq 1$.  In
section 5, we study the orbital stability of periodic traveling wave solutions of \eqref{gkdv} in neighborhoods
of the homoclinic and equilibrium solutions.  Section 6 is devoted to the application of
our theory to the case of the KdV.  In particular, it is shown that such solutions are orbitally
stable if and only if they are spectrally stable to perturbations of the same period as the underlying
wave.

\section{Properties of the Stationary Periodic Traveling Waves}

In this section, we recall the basic properties of the periodic traveling wave solutions of \eqref{gkdv}.
For each number $c>0$, a stationary traveling wave solution of \eqref{gkdv} with wave speed $c$
is a solution of the traveling wave ordinary differential equation
\begin{equation}
u_{xxx}+f(u)_x-cu_x=0,\label{travelode}
\end{equation}
i.e. they are solutions of \eqref{gkdv} which are stationary in the moving coordinate frame
defined by $x+ct$.  Clearly, such solutions are reducible to quadrature and satisfy
\begin{align}
u_{xx}+f(u)-cu&=a,\label{quad1}\\
\frac{1}{2}u_x^2+F(u)-\frac{c}{2}u^2-a u&=E,\label{quad2}
\end{align}
where $a$ and $E$ are real constants of integration, and $F$ satisfies $F'=f$, $F(0)=0$.
In order to ensure the existence of periodic orbits of \eqref{travelode}, we require that the effective
potential
\[
V(u;a,c)=F(u)-\frac{c}{2}u^2-au
\]
is of class $C^2(\RM)$ and has a local minimum.  Notice this places a restriction on the allowable parameter regime for our
problem.  This motivates the following definition.

\begin{df}
We define the set $\Omega\subset\RM^3$ to be the open set consisting of all triples $(a,E,c)$
such that $c>0$ and the solution $u(x)=u(x;a,E,c)$ of \eqref{quad2} is periodic and its orbit
in phase space lies below the separatrix.
\end{df}

\begin{rem}
Taking into account the translation invariance of \eqref{gkdv}, it follows that for each
$(a,E,c)\in\Omega$ we can construct a one-parameter family of periodic traveling wave
solutions of \eqref{gkdv}: namely
\[
u_{\xi}(x,t)=u(x+ct+\xi;a,E,c)
\]
where $\xi\in\RM$.  Thus, the periodic traveling waves of \eqref{gkdv} constitute a four
dimensional manifold of solutions.  However, outside of the null-direction of the linearized operator which
this generates, the added constant of integration does not play an important role in our theory.
In particular, we can mod out the continuous symmetry of \eqref{gkdv} by requiring all periodic
traveling wave solutions to satisfy the conditions
$u_x(0)=0$ and $V\;'(u(0))<0$.  As a result, each periodic solution of \eqref{travelode} is an even function
of the variable $x$ with a local maximum at $x=0$.
\end{rem}

\begin{rem}
Notice that $a$ and $E$ are conserved quantities of the ODE flow generated by \eqref{travelode}.  Moreover,
the classical solitary waves corresponding to solutions of \eqref{gkdv} satisfying $\lim_{x\to\pm\infty}u(x)=0$
correspond to $a=E=0$, and hence constitute a codimension two subset of the traveling wave solutions of \eqref{gkdv}.
It seems natural then to expect that the stability of periodic traveling wave solutions will involve
variations in these added parameters, just as the solitary wave stability index involves variations
in the one (modulo translation) free parameter $c$.
\end{rem}

Throughout this paper, we will always assume that our periodic traveling waves correspond to an $(a,E,c)$
within the open region $\Omega$, and that the roots $u_{\pm}$ of $E=V(u;a,c)$ with $V(u;a,c)<E$
for $u\in(u_{-},u_{+})$ are simple.  Moreover, we assume the potential $V$ does not have
a local maximum in the open interval $(u_-,u_+)$.  It follows that $u_{\pm}$ are $C^1$ functions of
$a$, $E$, $c$ on $\Omega$, and that $u(0)=u_{-}$.  Moreover, given $(a,E,c)\in\Omega$, we define the
period of the corresponding solution to be
\begin{equation}
T=T(a,E,c):= 2\int_{ u_{-}}^{u_{+}} \frac{du}{\sqrt{2\left(E-V(u;a,c)\right)}}.\label{period}
\end{equation}
The above interval can be regularized at the square root branch
points $ u_-,\; u_{+}$ by the following procedure\footnote{In section 5 we will use a slightly different regularization which
is well-adapted to the KdV equation.}:
Write $E-V(u;a,c)=(u- u_{-})(u_{+}-u)Q(u)$ and consider
the change of variables $ u= \frac{u_++u_-}{2} + \frac{u_+-u_-}{2} \sin(\theta)$.  Notice
that $Q(u)\neq0$ on $[ u_-, u_{+}]$.  It follows that
$du=\sqrt{(u- u_{-})(u_{+}-u)}d\theta$ and hence
\eqref{period} can be written in a regularized form as
\[
T(a,E,c)=2\int_{-\frac{\pi}{2}}^{\frac{\pi}{2}}\frac{d\theta}{\sqrt{Q\left(\frac{u_++u_-}{2} + \frac{u_+-u_-}{2} \sin(\theta)\right)}}.
\]
In particular, we can differentiate the above relation with respect to the parameters $a$, $E$, and $c$ within
the parameter regime $\Omega$.  Similarly the mass, momentum, and Hamiltonian of the traveling wave are given
by 
\begin{align}
M(a,E,c) &= \langle u\rangle = \int_0^T u(x)\; dx = 2\int_{u_-}^{u_+} \frac{u\; du}{\sqrt{2\left(E-V(u;a,c)\right)}}\label{mass} \\
P(a,E,c) &= \langle u^2\rangle = \int_0^T u^2(x)\; dx = 2\int_{u_-}^{u_+} \frac{u^2\; du}{\sqrt{2\left(E-V(u;a,c)\right)}}\label{momentum} \\
H(a,E,c) &=  \left\langle \frac{u_x^2}{2} - F(u)\right\rangle =  2\int_{u_-}^{u_+} \frac{E-V(u;a,c) - F(u)}{\sqrt{2\left(E-V(u;a,c)\right)}}\;du\label{energy} .
\end{align}
Notice that these integrals can be regularized as above, and represent conserved
quantities of the gKdV flow restricted to the manifold of periodic traveling
wave solutions. In particular one can differentiate the above expressions with respect to the parameters $(a,E,c)$.

Throughout this paper, a large role will be played by the gradients of the above
conserved quantities.  However, by the Hamiltonian structure of \eqref{travelode},
the corresponding derivatives of the period, mass, and momentum restricted
to a periodic traveling wave $u(\cdot;a,E,c)$ with $(a,E,c)\in\Omega$
satisfy several useful identities.  In particular if we define the classical action
\begin{equation}
K = \oint u_x \;du = \int_0^T u_x^2 \;dx = 2\int_{u_-}^{u_+}\sqrt{2(E-V(u;a,c))}\;du \nonumber
\end{equation}
(which is not itself conserved) then this quantity satisfies the following relations
\begin{align*}
K_E(a,E,c) &= T(a,E,c) \\
K_a(a,E,c) &= M(a,E,c) \\
2K_c(a,E,c) &= P(a,E,c).
\end{align*}
Moreover, using the fact that $T,M,P$
and $H$ are $C^1$ functions of parameters $(a,E,c)$, the above implies the
following relationship between the gradients of the conserved quantities
of the gKdV flow restricted to the periodic traveling waves:
\begin{equation}
E \nabla T + a \nabla M + \frac{c}{2} \nabla P + \nabla H = 0,\label{overdetermined}
\end{equation}
where $\nabla = (\partial_a,\partial_E,\partial_c)$: see the appendix
of \cite{BrJ} for details of this calculation.  The subsequent theory is developed most naturally
in terms of the quantities $T$, $M$, and $P$.  However, it is possible to restate our results
in terms of $M$, $P$, and $H$ using the identity \eqref{overdetermined}.
This is desirable since these have a natural interpretation as conserved
quantities of the partial differential equation \eqref{gkdv}.

We now discuss the parametrization of the periodic solutions of \eqref{travelode} in more
detail.  A major technical necessity throughout this paper is that the constants
of motion for the PDE flow defined by \eqref{gkdv} provide (at least locally) a good
parametrization for the periodic traveling wave solutions.  In particular, we assume
for a given $(a,E,c)\in\Omega$ the conserved quantities
$(H,M,P)$ are good local coordinates for the periodic traveling waves near $(a,E,c)$.
More precisely, we assume the map $(a,E,c)\to(H(a,E,c),M(a,E,c),P(a,E,c))$ has a unique
$C^1$ inverse in a neighborhood of $(a,E,c)$.  If we adopt the notation
\[
\{f,g\}_{x,y} = \left|\begin{array}{cc} f_x & g_x \\ f_y & g_y \end{array}\right|
\]
for $2 \times 2$ Jacobians, and $\{f,g,h\}_{x,y,z}$ for the analogous $3\times 3$ Jacobian,
it follows this is possible exactly when $\{H,M,P\}_{a,E,c}$ is non-zero, which is equivalent to
$\{T,M,P\}_{a,E,c}\neq 0$ by \eqref{overdetermined}.  Also, we will need to know that two
of quantities $T$, $M$ and $P$ 
provide a local parametrization for the traveling waves with fixed wave speed.  By reasoning
as above, this happens exactly when the matrix
\[
\left(\begin{array}{ccc} T_a & M_a & P_a \\ T_E & M_E & P_E \end{array}\right)
\]
has full rank.  A sufficient requirement is thus $\{T,M\}_{a,E}\neq 0$: trivial modifications
of our theory is needed if $\{T,M\}_{a,E}=0$ but one of the other sub-determinants
do not vanish.

\section{Spectral Stability Analysis}

In this section, we recall the relevant results of \cite{BrJ} on the spectral stability of periodic-traveling wave solutions of
the gKdV.  Suppose that $u=u(\;\cdot\;;a,E,c)\in C^3(\RM;\RM)$ is a $T$-periodic periodic solution of \eqref{travelgkdv}.  Linearizing about
this solution and taking the Fourier transform in time leads to the spectral problem
\begin{equation}
\partial_x\mathcal{L}[u]v=\mu v\label{linearization}
\end{equation}
considered on $L^2(\RM;\RM)$, where $\mathcal{L}[u]:=-\partial_x^2-f'(u)+c$ is a closed symmetric linear operator with periodic coefficients.
In particular, since $u$ is bounded it follows that $\mathcal{L}[u]$ is in fact a self-adjoint operator on $L^2(\RM)$ with
densely defined domain $C^\infty(\RM)$.  Notice that considering \eqref{linearization} on $L^2(\RM)$ corresponds
to considering the spectral stability of $u$ with respect to localized perturbations.  One could also
study the spectral stability with respect to uniformly bounded perturbations, but by standard results
in Floquet theory the resulting theories are equivalent.  In order to understand the nature
of the spectrum of \eqref{linearization} we make the following definition.

\begin{df}
The monodromy matrix $\MM(\mu)$ is defined to be the period map
\[
\MM(\mu)=\Phi(T,\mu)
\]
where $\Phi(x,\mu)$ satisfies the first order system
\begin{equation}
\Phi_x=\HM(x,\mu)\Phi,\;\;\Phi(0,\mu)={\bf I}\label{system}
\end{equation}
where ${\bf I}$ is the $3\times 3$ identity matrix and
\[
\HM(x,\mu)=\left(
           \begin{array}{ccc}
             0 & 1 & 0 \\
             0 & 1 & 0 \\
             -\mu-u_xf''(u) & -f'(u)+c & 0 \\
           \end{array}
         \right).
\]
\end{df}

It follows from a straightforward calculation that the operator $\partial_x\mathcal{L}[u]$ has no $L^2$ eigenvalues.  Indeed,
from Floquet's theorem we know that every solution of \eqref{system} can be written in the form $\Phi(x)=e^{\omega x}p(x)$
where $p$ is a $T$-periodic function and $\lambda=e^{\omega T}$ is an eigenvalue of $\MM(\mu)$.  Conversely, each eigenvalue of
$\MM(\mu)$ leads to a solution of this form.  It then follows that for any $N\in\mathbb{N}$ one has
\[
\Phi(NT)=e^{\omega N T}p(0)=\lambda^Np(0).
\]
If $\Phi(x)$ decays as $x\to\infty$, the above relation implies that it must be unbounded as $x\to -\infty$, from which
our claim follows.  In particular, a solution $\Phi(x)$ of \eqref{system} can be at most bounded on $\RM$.  This
leads us to the following definition.

\begin{df}
We say $\mu\in\spec\left(\partial_x\mathcal{L}[u]\right)$ is there exists a non-trivial bounded function
$\psi$ such that $\partial_x\mathcal{L}[u]\psi=\mu\psi$ or, equivalently, if there exists a $\lambda\in\CM$
with $|\lambda|=1$ such that
\[
\det\left(\MM(\mu)-\lambda{\bf I}\right)=0.
\]
Following Gardner \cite{Gard}, we define the periodic Evans function $D:\CM\times\CM\to\CM$ to be
\[
D(\mu,\lambda)=\det\left(\MM(\mu)-\lambda{\bf I}\right).
\]
Moreover, we say the periodic solution $u(x;a,E,c)$ is spectrally stable if $\spec\left(\partial_x\mathcal{L}[u]\right)$ is a
subset of the imaginary axis.
\end{df}

\begin{rem} In general, one defines a solution $u(x;a,E,c)$ to be spectrally stable if $\spec(\partial_x\mathcal{L}[u])$
does not intersect the open right half plane.  However, the Hamiltonian nature of the problem implies the spectrum
is symmetric with respect to the real and imaginary axis, from which our definition follows.  Notice that since
we primarily concerned with roots of $D(\mu,\lambda)$ with $\lambda$ on the unit circle, we will frequently
work with the function $D(\mu,e^{i\kappa})$ for $\kappa\in\RM$.  However, the more general definition given above
is useful for certain analyticity arguments.
\end{rem}

It follows that the set $\spec(\partial_x\mathcal{L}[u])$ consists of precisely the $L^\infty$ eigenvalues
of the linearized operator $\partial_x\mathcal{L}[u]$.  Moreover, if we define a projection operator
$\pi_1:\CM\times\CM\to\CM$ by $\pi_1(z_1,z_2)=z_1$ for $(z_1,z_2)\in\CM\times\CM$, then
the projection of the zero set of $D(\mu,e^{i\kappa})$ in $\CM\times S^1$ via $\pi_1$
is precisely the set $\spec(\partial_x\mathcal{L}[u])$.  In particular we see that $D(\mu,1)$
detects spectra which corresponds to perturbations which are periodic with the same period
as the underlying solution $u$.  If the zero set of $D(\mu,1)$ is a subset
of the imaginary axis, it follows that the underlying periodic solution is spectrally stable with respect to
perturbations of the same period.  Although this is a long way from concluding any sense of nonlinear
stability, we find this analysis is vital to understanding the orbital stability calculation.
In particular, we recall the following easily proved result.

\begin{lem}
The function $D(\mu,1)$ is odd and the limit $\lim_{\mu\to\infty}{\rm sign}\left(D(\mu,1)\right)$ exists and is negative.
\end{lem}

The proof relies on an asymptotic formula for the quantity $\tr(\MM(\mu))$
for $\mu\gg 1$, as well as a structural property of the Evans function which holds due to the Hamiltonian nature
of the problem.  For details, see Proposition $1$ and Lemma $1$ of \cite{BrJ}.  In order to determine
if $D(\mu,1)$ has non-zero real roots, we also need the following technical lemma.

\begin{lem}
The periodic Evans function satisfies the following asymptotic relation near $\mu=0$:
\[
D(\mu,1)=-\frac{1}{2}\{T,M,P\}_{a,E,c}\;\mu^3+\mathcal{O}(|\mu|^4).
\]
\end{lem}

The proof of this lemma is somewhat technical, although straightforward, and will not be reproduced here: see Theorem 3
in \cite{BrJ} for details.  With the above two lemmas, it immediately follows that if $\{T,M,P\}_{a,E,c}$
is negative, then the number of positive roots of $D(\mu,1)$ is odd and hence one has exponential
instability of the underlying periodic traveling wave.  Moreover, we will show in Lemma \ref{periodLem}
that $T_E\geq 0$ implies $\mathcal{L}[u]$ has exactly one negative eigenvalue.  It follows from that
the linearized operator $\partial_x\mathcal{L}[u]$ has at most one unstable eigenvalue with positive real part, counting
multiplicities (see Theorem 3.1 of \cite{PW1}).  Since the spectrum of $\partial_x\mathcal{L}[u]$ is symmetric about the
real and imaginary axis, it follows that all unstable periodic eigenvalues of the
linearized operator must be real.  This proves the following extension of Corollary 1 in \cite{BrJ}.

\begin{thm}\label{orientation}
Let $u(x;a_0,E_0,c_0)$ be a periodic traveling wave solution of \eqref{gkdv}.  
If $T_E(a_0,E_0,c_0)>0$, then the solution is spectrally stable to perturbations
of the same period if and only if $\{T,M,P\}_{a,E,c}$ is positive at $(a_0,E_0,c_0)$.
\end{thm}

Notice that if $T_E<0$ the operator $\mathcal{L}[u]$ has two negative eigenvalues and hence if $\{T,M,P\}_{a,E,c}>0$
there is no way of proving from these methods whether the number of periodic eigenvalues
of $\partial_x\mathcal{L}[u]$ with positive real part is equal to zero or two.
In the next section, we will prove that, if $\{T,M,P\}_{a,E,c}$ is positive,
then the periodic traveling wave $u(x;a,E,c)$ is
orbitally stable if and only if the sub-determinants $T_E$ and $\{T,M\}_{a,E}$ are
also positive.  We will show these sub-determinants are positive in certain cases, but we do not
know if this is true in general.  We conjecture, however, that this condition holds for all $(a,E,c)\in\Omega$
in the case of a power-law nonlinearity.

\section{Orbital Stability}

In this section, we prove our main theorem on the orbital stability of periodic traveling wave solutions
of \eqref{gkdv}.  Throughout this section, we assume we have a $T$-periodic traveling wave solution
$u(x;a_0,E_0,c_0)$ of equation \eqref{gkdv}, i.e. we assume $u$ satisfies
\begin{equation}
\frac{1}{2}u_x^2+F(u)-\frac{c_0}{2}u^2-a_0 u=E_0\label{quad0}
\end{equation}
with $(a_0,E_0,c_0)\in\Omega$ and $T=T(a_0,E_0,c_0)$.  Moreover, we assume the non-linearity $f$ present in \eqref{gkdv} is such
that the Cauchy problem for \eqref{travelgkdv} is globally well-posed in a real Hilbert space $X$
of real valued $T$ periodic functions defined on $\RM$, which we equip with the standard $L^2([0,T])$ inner product
\[
\left<g,h\right>:=\int_0^Tg(x)h(x)dx
\]
for all $g,h\in X$, and we identify the dual space $X^*$ through the pairing
\[
\left<g,h\right>_*=\int_0^Tg(x)h(x)dx
\]
for all $g\in X^*$ and $h\in X$.  In particular, notice that $L^2([0,T])$ is required to be a
subspace of $X$.  For example, if $f(u)=u^3/3$, corresponding to the
modified Korteweg-de Vries equation, then the Cauchy problem for \eqref{travelgkdv} is
globally well-posed in the space
\[
H_{\rm per}^s([0,T];\RM)=\{g\in H^s([0,T];\RM):g(x+T)=g(x)\;a.e.\}
\]
for all $s\geq \frac{1}{2}$, where we identify the dual space with $H_{\rm per}^{-s}([0,T];\RM)$ through the
above pairing (see \cite{CKSTT} for proof).  Moreover, due to the structure of the
gKdV, we make the natural assumption that the evolution of \eqref{travelgkdv} in the space $X$ is invariant
under a one parameter group of isometries $G$ corresponding to spatial translation.  Thus, $G$ can be identified
with the real line $\RM$ acting on the space $X$ through the unitary representation
\[
(R_{\xi}g)(x)=g(x+\xi)
\]
for all $g\in X$ and $\xi\in G$.  Since the details
of our proof works regardless of the form of the non-linearity $f$, we make the above additional assumptions on
the nonlinearity and make no other references to the exact structure of the space $X$ nor $f$.

In view of the symmetry group $G$, we now describe precisely what we mean by orbital stability.
We define the G-orbit generated by $u$ to be
\[
\mathcal{O}_{u}:=\{R_\xi u:\xi\in G\}.
\]
Now, suppose we have initial data $\phi_0\in X$ which is close to the orbit $\mathcal{O}_{\varepsilon}$.  By
orbital stability, we mean that if $\phi(\cdot,t)\in X$ is the unique solution with initial data $\phi_0$, then
$\phi(\cdot,t)$ is close to the orbit of $u$ for all $t>0$.  More precisely, we introduce a semi-distance $\rho$
defined on the space $X$ by
\[
\rho(g,h)=\inf_{\xi\in G}\|g-R_{\xi}h\|_{X},
\]
and use this to define an $\varepsilon$-neighborhood of the orbit $\mathcal{O}_{u}$ by
\[
\mathcal{U}_{\varepsilon}:=\{\phi\in X:\rho(u,\phi)<\varepsilon\}.
\]
The main result of this section is the following restatement of Theorem \eqref{os:state}.

\begin{prop}\label{os}
Let $u(x)=u(x;a_0,E_0,c_0)$ solve \eqref{quad0} and suppose the quantities $T_E$, $\{T,M\}_{a,E}$, and $\{T,M,P\}_{a,E,c}$ are positive
at $(a_0,E_0,c_0)$.  Then there exists positive constants $C_0,\varepsilon_0$ such that
if $\phi_0\in X$ satisfies $\rho(\phi_0,u)<\varepsilon$ for some $\varepsilon<\varepsilon_0$, then the solution $\phi(x,t)$
of \eqref{gkdv} with initial data $\phi_0$ satisfies $\rho(\phi(\cdot,t),u)\leq C_0\varepsilon$.
\end{prop}

\begin{rem}\label{rem:unstable}
Notice that Theorem \ref{orientation} implies a periodic solution $u(x;a_0,E_0,c_0)$ of \eqref{travelode}
is an exponentially unstable solution of \eqref{gkdv} if $\{T,M,P\}_{a,E,c}$ is negative at $(a_0,E_0,c_0)$.
Thus, the positivity of this Jacobian is a necessary condition for nonlinear stability.
\end{rem}

Before we prove Proposition \ref{os} we wish to shed some light on the hypotheses.  Recall that the classical
action $K(a,E,c)$ of the periodic traveling wave satisfies
\[
\nabla_{a,E,c}K(a,E,c)=\left(M(a,E,c),T(a,E,c),\frac{1}{2}P(a,E,c)\right).
\]
As a result, we can write its Hessian as
\[
D^2_{a,E,c}K(a,E,c)=\left(
                      \begin{array}{ccc}
                        M_a & M_E & M_c \\
                        T_a & T_E & T_c \\
                        P_a & P_E & P_c \\
                      \end{array}
                    \right).
\]
Proposition \ref{os} thus states that if $(a_0,E_0,c_0)\in\Omega$, the corresponding periodic traveling
wave solutions of \eqref{gkdv} is orbitally stable if the principle minor determinants of $D^2_{a,E,c}K(a,E,c)$ satisfy
$d_1=T_E>0$, $d_2=\{M,T\}_{a,E}<0$, and $d_3=\{M,T,P\}_{a,E,c}<0$.   It is clear that a necessary condition
for this claim is that the Hessian $D^2_{a,E,c}K(a,E,c)$ is invertible with precisely one negative eigenvalue.
However, this is clearly not sufficient.  Although we are unable to give a complete characterization of the
orbital stability in terms of orthogonal invariants of $D^2_{a,E,c}K(a,E,c)$, such a connection
would indeed be very interesting.

We now proceed with the proof of Proposition \ref{os}, which follows the general method of Bona, Souganidis and Strauss \cite{BSS}, and
consequently that of Grillakis, Shatah and Strauss \cite{GSSI},\cite{GSSII}.
We define the following functionals on the space $X$, which correspond to the
``energy", ``mass" and ``momentum" respectively:
\begin{align*}
\mathcal{E}(\phi)&:=\int_0^T\left(\frac{1}{2}\phi_x(x)^2-F(\phi(x))\right)dx\\
\mathcal{M}(\phi)&:=\int_0^T\phi(x)\;dx\\
\mathcal{P}(\phi)&:=\frac{1}{2}\int_0^T\phi(x)^2dx.
\end{align*}
These functionals represent conserved quantities of the flow generated
by \eqref{gkdv}.  In particular, if $\phi(x,t)$ is a solution of \eqref{gkdv} of
period $T$, then the quantities $\mathcal{E}(\phi(\cdot,t))$, $\mathcal{M}(\phi(\cdot,t))$,
and $\mathcal{P}(\phi(\cdot,t))$ are constants in time.  Moreover,
notice that $\mathcal{E}(u)=H(a_0,E_0,c_0)$, $\mathcal{M}(u)=M(a_0,E_0,c_0)$,
and $\mathcal{P}(u)=P(a_0,E_0,c_0)$ where $H$, $M$ and $P$ are defined in \eqref{mass}-\eqref{energy}.

\begin{rem}
Throughout the remainder of this paper, the symbols $M$ and $P$ will denote the functionals
$\mathcal{M}$ and $\mathcal{P}$ restricted to the manifold of periodic traveling wave solutions
of \eqref{gkdv} with $(a,E,c)\in\Omega$.
\end{rem}

\begin{rem}\label{kdv:rem}
Calculations in similar vain have been carried out recently in the special cases of cnoidal
solutions of the KdV \cite{PBS}, as well as for traveling wave solutions of the modified KdV
arising from \eqref{gkdv} with $f(u)=u^3$ \cite{Pava}.  In each of these cases, however, it was
assumed that $a=0$, or equivalently that $M(a,E,c)=0$.  While this is always possible
for the KdV (due to Galilean invariance), this is not possible for general nonlinearities
without restricting your admissible class of traveling wave solutions, i.e. restricting $\Omega$.
As we are interested in deriving universal conditions for stability of traveling wave solutions of \eqref{gkdv},
we are forced to work with all three functionals defined above.
\end{rem}

It is easily verified that $\mathcal{E}$, $\mathcal{M}$ and $\mathcal{P}$
are smooth functionals on $X$, whose first derivatives are smooth maps from $X$ to $X^*$ defined by
\[
\mathcal{E}'(\phi)=-\phi_{xx}-f(\phi),\;\;\mathcal{M}'(\phi)=1,\;\;\mathcal{P}'(\phi)=\phi.
\]
If we now define an augmented energy functional on the space $X$ by
\begin{equation}
\mathcal{E}_0(\phi):=\mathcal{E}(\phi)+c_0\mathcal{P}(\phi)+a_0\mathcal{M}(\phi)+E_0 T\label{energy2}
\end{equation}
it follows from \eqref{quad0} that $\mathcal{E}_0(u)=0$ and $\mathcal{E}_0'(u)=0$.  Hence, $u$ is a critical point of the
functional $\mathcal{E}_0$.

\begin{rem}
Notice that the added factor of $E_0 T$ on the right hand side of \eqref{energy2} is not technically
needed for our calculation.  However, we point out that (formally) if we consider variations
in $\mathcal{E}_0$ in the {\it period} we obtain
\begin{align*}
\frac{\partial}{\partial T}\mathcal{E}_0(\phi)\big{|}_{\phi=u}
      &=\frac{1}{2}u_x^2(T)-F(u(T))+au(T)+E+\left<\mathcal{E}_0'(u),\frac{\partial u}{\partial T}\right>\\
      &=\frac{1}{2}u_x^2(T)+\frac{1}{2}u_x^2(T)+\left<\mathcal{E}_0'(u),\frac{\partial u}{\partial T}\right>\\
&=0
\end{align*}
since $u_x(T)=0$ and $\mathcal{E}_0'(u)=0$.
Hence $u$ is also (formally) a critical point of the modified energy with respect to variations in the period.
It would be very interesting to try to make this calculation rigorous and to see if it allows
one to extend orbital stability results to include perturbations with period {\rm close} to the period
of the underlying periodic wave.  Again, this is clearly all very formal and we will make no
attempt at such a theory here.
\end{rem}

To determine the nature of this critical point, we consider its second
derivative $\mathcal{E}_0''$, which is a smooth map from $X$ to $\mathcal{L}(X,X^*)$ defined by
\[
\mathcal{E}_0''(\phi)=-\phi_{xx}-f'(\phi)+c_0.
\]
This formula immediately follows by noticing the second derivatives of the mass, momentum, and energy
functionals are smooth maps from $X$ to $\mathcal{L}(X,X^*)$ defined by
\[
\mathcal{E}''(\phi)=-\partial_x^2-f'(\phi),\;\;\mathcal{M}''(\phi)=0,\;\;\mathcal{P}''(\phi)=1.
\]
In particular, notice the second derivative of the augmented energy functional $\mathcal{E}_0$ at the
critical point $u$ is precisely linear operator $\mathcal{L}[u]$
arising from linearizing \eqref{travelgkdv} with wave speed $c_0$ about $u$.
It follows from the comments in the previous section that $\mathcal{E}_0''(u)$ is a self-adjoint
linear operator on $L_{\rm per}^2([0,T];\RM)$ with compact resolvent.
In order to classify $u$ as a critical point of $\mathcal{E}_0$, we must understand
the nature of the spectrum of the second variation $\mathcal{L}[u]$: in particular, we need
to know the number of negative eigenvalues.  This is handled in the following lemma.

\begin{lem}\label{specL}
The spectrum of the operator $\mathcal{L}[u]$ considered on the space $L_{\rm per}^2([0,T])$ satisfies the following
trichotomy:
\begin{enumerate}
  \item[(i)] If $T_E>0$, then $\mathcal{L}[u]$ has exactly one negative eigenvalue, a simple eigenvalue at zero, and the
  rest of the spectrum is strictly positive and bounded away from zero.
  \item[(ii)] If $T_E=0$, then $\mathcal{L}[u]$ has exactly one negative eigenvalue, a double eigenvalue at zero, and the
  rest of the spectrum is strictly positive and bounded away from zero.
  \item[(iii)] If $T_E<0$, then $\mathcal{L}[u]$ has exactly two negative eigenvalues, a simple eigenvalue at zero, and the
  rest of the spectrum is strictly positive and bounded away from zero.
\end{enumerate}
\end{lem}

\begin{proof}
This is essentially a consequence of the translation invariance of \eqref{gkdv} and the Strum-Liouville oscillation theorem.
Indeed, notice that for any $\xi\in G$ the function $R_\xi u$ is a stationary solution of \eqref{quad1} with
wave speed $c_0$ and $a=a_0$.  Differentiating this relation with respect to $\xi$ and evaluating at $\xi=0$
implies that $\mathcal{L}[u]u_x=0$.  Moreover, since $u$ is radially increasing on $[0,T]$
from its local minimum there, $u_x$ is periodic with the same period as $u$ and hence $u_x\in L_{\rm per}^2([0,T])$.
This proves that zero is always a periodic eigenvalue of $\mathcal{L}[u]$ as claimed.
To see there is exactly one negative eigenvalue, notice that
since $u$ is $T$-periodic with precisely one local critical point on $(0,T)$, its derivative $u_x$ must have precisely
one sign change over its period.  By standard Strum-Liouville theory applied to the periodic problem (see Theorem 2.14 in \cite{MW}),
it follows that zero must be the either the second or third\footnote{Clearly, we mean with respect to the natural ordering on $\RM$.}
eigenvalue of $\mathcal{L}[u]$ considered on the space $L_{\rm per}^2(\RM)$.

Next, we show that zero is a simple eigenvalue of $\mathcal{L}[u]$ on the space $L_{\rm per}^2([0,T])$ if
and only if $T_E\neq 0$.  To this end,
notice that the periodic traveling wave solutions of \eqref{quad1} are invariant under changes in the energy parameter
$E$ associated with the Hamiltonian ODE \eqref{travelode}.
As above, it follows that $\mathcal{L}[u]u_E=0$.  We must now determine whether the function $u_E$ belongs
to the space $L_{\rm per}^2([0,T])$.  Since it is clearly smooth, we must only check whether it is periodic
with the same period as the underlying wave $u$.  To this end, we use $u_x$ and $u_E$ as a basis to compute the monodromy
matrix $\mm(0)$ corresponding to the equation $\mathcal{L}[u]v=0$.  Notice that
differentiating the relation $E=V(u_{-};a,c)$ with respect to $E$ and evaluating at $(a_0,E_0,c_0)$ gives
$\frac{\partial u_{-}}{\partial E}V'(u_{-};a_0,c_0)=1$, and hence
$\frac{\partial u_{-}}{\partial E}$ is non-zero at $(a_0,E_0,c_0)$.
Defining $y_1(x)=\left(\frac{d u_{-}}{dE}\right)^{-1}u_{E}$
and $y_2(x)=-\left(V'(u_{-};a_0,c_0)\right)^{-1}u_{x}(x)$, it follows from direct calculation that
\[
\begin{array}{cc}
  y_1(0)=1, & y_2(0)=0, \\
  y_1'(0)=0, & y_2'(0)=1.
\end{array}
\]
Thus, it follows by calculating $u_{E}(T)$ by the chain rule that we have
\[
\mm(0)=\left(
         \begin{array}{cc}
           1 & T_E\\
           0 & 1 \\
         \end{array}
       \right),
\]
where again we have used the fact that $V'(u_{-};a,c)\frac{\partial u_{-}}{\partial E}=1$.  Thus, it follows
that zero is a simple eigenvalue of $\mathcal{L}[u]$ if and only if $T_E\neq 0$, and that the multiplicity
will be two in the case $T_E=0$.

Finally, to determine whether $\mu=0$ is the second or third eigenvalue of $\mathcal{L}[u]$, we note
that from the results of \cite{BrJ} we have
\begin{equation}
{\rm sign}(T_E)={\rm sign}\left(\tr(\mm_{\mu}(0))\right).\label{krein}
\end{equation}
By the oscillation theorem for Hill-operators with periodic coefficients (see Theorem 2.1 in \cite{MW}),
it follows that $\mu=0$ is the second periodic eigenvalue of $\mathcal{L}[u]$ if and only if $T_E\geq 0$, and the third
periodic eigenvalue if and only if $T_E<0$.  This completes the proof.
\end{proof}

Notice that in the solitary wave case, the spectrum of the operator $\mathcal{L}[u]$ {\it always}
satisfies $(i)$ in the above trichotomy.  Since $E$ is not restricted to be zero in the periodic
context, it is not surprising that such a non-trivial trichotomy might exist.
The next lemma shows that for a large class of nonlinearities, the period is indeed an increasing function
of $E$ within the region $\Omega$.  

\begin{lem}\label{periodLem}
Let $(a_0,E_0,c_0)\in\Omega$ and $u=u(\;\cdot\;;a_0,E_0,c_0)$ denote the corresponding periodic solution
of \eqref{travelode} with wave speed $c_0$ and period $T=T(a_0,E_0,c_0)$.  If the nonlinearity $f$ in \eqref{gkdv}
is such that $f'(u)$ is co-periodic with $u$, then $T_E>0$ at $(a_0,E_0,c_0)$.
\end{lem}

\begin{proof}
If $f'(u)$ is co-periodic with $u$, then the operator $\mathcal{L}[u]$ is a Hill operator with
even potential with period $T=T(a_0,E_0,c_0)$.  Thus, $u_x$ is a periodic eigenvalue of $\mathcal{L}[u]$
which satisfies Dirichlet boundary conditions.  Moreover, since $u_x$ changes signs once over
$(0,T)$, it follows that zero is either the second or third periodic eigenvalue of $\mathcal{L}[u]$.
Since the first periodic eigenvalue must be even, and hence must satisfy Neumann boundary conditions, it follows
by Dirichlet-Neumann bracketing \cite{RS} that zero must be the second periodic eigenvalue of $\mathcal{L}[u]$.
Using the notation of Lemma \ref{specL}, it follows that $\tr(\mm_{\mu}(0))>0$ and hence $T_E>0$
by equation \eqref{krein}.
\end{proof}

\begin{rem}
In particular, it follows that if $f(u)=u^{p+1}$ for some $p\geq 1$
then the spectrum of $\mathcal{L}[u]$ will satisfy ({\it i}) in Lemma \ref{specL}.
\end{rem}

Throughout the rest of the paper, unless otherwise stated, we will assume that $T_E>0$ at $(a_0,E_0,c_0)$ and
hence zero is a simple eigenvalue of the operator $\mathcal{L}[u]$
considered on the space $L_{\rm per}^2(\RM)$.  In particular, we assume that the map $E\to T(E,a_0,c_0)$ does
not have a critical point at $E_0$.  It follows that we can define the spectral projections $\Pi_{-}$, $\Pi_{0}$
and $\Pi_{+}$ onto the negative, zero, and positive subspaces of the operator $\mathcal{L}[u]$
(respectively) via the Dunford calculus.  Thus, any $\phi\in X$ can be decomposed as a linear
combination of 
$u_x$, an element in the positive subspace of $\mathcal{L}[u]$,
and $\chi$, where $\chi$ is the unique positive eigenfunction of $\mathcal{L}[u]$ with $\|\chi\|_{L^2([0,T])}=1$
which satisfies
\[
\left<\mathcal{L}[u]\chi,\chi\right>=-\lambda^2
\]
for some $\lambda>0$.  From the above definition of $\chi$ it follows that $\chi$ is the eigenfunction corresponding to the
unique negative eigenvalue $-\lambda^2$ of $\mathcal{L}[u]$.

From Lemma \ref{specL}, we know that $u$ is a degenerate saddle point of the functional $\mathcal{E}_0$
on $X$, with one unstable direction and one neutral direction.  In order to get rid of the unstable direction, we note that
the evolution of \eqref{gkdv} does not occur on the entire space $X$, but on the co-dimension two submanifold defined by
\[
\Sigma_0:=\{\phi\in X:\mathcal{M}(\phi)=M(a_0,E_0,c_0),\;\mathcal{P}(\phi)=P(a_0,E_0,c_0)\}.
\]
It is clear that $\Sigma_0$ is indeed a smooth submanifold of $X$ in a neighborhood of the group orbit $\mathcal{O}_u$.
Moreover, the entire orbit $\mathcal{O}_u$ is contained in $\Sigma_0$.  The main technical
result needed for this section is that the functional $\mathcal{E}_0$ is coercive on $\Sigma_0$ with
respect to the semi-distance $\rho$, which is the content of the following proposition.

\begin{prop}\label{coercive}
If each of the quantities $T_E$, $\{T,M\}_{a,E}$, and $\{T,M,P\}_{a,E,c}$ are positive, then there exists positive
constants $C_1$, $\delta$ which depend on $(a_0,E_0,c_0)$ such that
\[
\mathcal{E}_0(\phi)-\mathcal{E}_0(u)\geq C_1\rho(\phi,u)^2
\]
for all $\phi\in\Sigma_0$ such that $\rho(\phi,u)<\delta$.
\end{prop}

The proof of Proposition \ref{coercive} is broken down into three lemmas which analyze the quadratic
form induced by the self adjoint operator $\mathcal{L}[u]$.  
To begin, we define a function $\phi_0$ by
\[
\phi_0(x):=\{u(x;a,E,c),T(a,E,c),M(a,E,c)\}_{a,E,c}\big{|}_{(a_0,E_0,c_0)}.
\]
It follows from a straightforward calculation (see Proposition 4 in \cite{BrJ}) that $\phi_0\in X$ and
\[
\mathcal{L}[u]\phi_0=\{T,M\}_{E,c}-\{T,M\}_{a,E}u,
\]
where the right hand side is evaluated at $(a_0,E_0,c_0)$.  This function plays a large role in the
spectral stability theory for periodic traveling wave solutions\footnote{Actually, the function $u_c$
plays a large role in our analysis via the periodic Evans function.  However, since $u_c$ is not
in general $T$-periodic due to the dependence of the period on the wave speed, we work
here with its periodic analogue $\phi_0$.} of \eqref{gkdv} outlined in section 3.
In particular, we have $\partial_x\mathcal{L}[u]\phi_0=-\{T,M\}_{a,E}u_x$, and hence, assuming
$\{T,M\}_{a,E}\neq 0$ at $(a_0,E_0,c_0)$, $\phi_0$ is
in the generalized periodic null space of the linearized operator $\partial_x\mathcal{L}[u]$.
See \cite{BrJ} for more details on the role of this function in the spectral stability
theory outlined in section 2.

Notice that $\phi_0$ does not belong to the set
\[
\mathcal{T}_0=\{\phi\in X:\left<u,\phi\right>=\left<1,\phi\right>=0\},
\]
which is precisely the tangent space in $X$ to $\Sigma_0$.  Indeed, while $\left<1,\phi_0\right>=0$
by construction, the inner product $\left<u,\phi_0\right>=\{T,M,P\}_{a,E,c}$ does not
vanish by hypothesis.  Using the spectral resolution of the operator $\mathcal{L}[u]$,
we begin the proof of Proposition \ref{coercive} with the following lemma.

\begin{lem}\label{positivedefinite}
Assume that $T_E\geq0$ and $\{T,M\}_{a,E}>0$.  Then
\[
\left<\mathcal{L}[u]\phi,\phi\right>>0.
\]
for every $\phi\in\mathcal{T}_0$ which is orthogonal to the periodic
null space of $\mathcal{L}[u]$.
\end{lem}

\begin{proof}
The proof is essentially found in \cite{BSS}.  First, suppose that $T_E>0$ and note that by Lemma \ref{specL} we can write
\begin{align*}
\phi_0&=\alpha\chi+\beta u_x+p\\
\phi&=A\chi+\widetilde{p}
\end{align*}
for some constants $\alpha$, $\beta$, $A$, and functions $p$ and $\widetilde{p}$ belonging
to the positive subspace of $\mathcal{L}[u]$.  By assumption the quantity
\begin{equation}
\left<\mathcal{L}[u]\phi_0,\phi_0\right>=-\{T,M\}_{a,E}\{T,M,P\}_{a,E,c}\label{osindex}
\end{equation}
is negative, and hence the above decomposition of $\phi_0$ implies that
\begin{equation}
0>\left<-\lambda^2\alpha\chi+\mathcal{L}[u]p,\alpha\chi+\beta u_x+p\right>=-\lambda^2\alpha^2+\left<\mathcal{L}[u]p,p\right>,
\label{ineq:1}
\end{equation}
which gives an upper bound on the positive number $\left<\mathcal{L}[u]p,p\right>$.  Similarly,
using the above decomposition of $\phi$ we have
\begin{equation}
0=\left<\mathcal{L}[u]\phi_0,\phi\right>=-\lambda^2 A\alpha+\left<\mathcal{L}[u]p,\widetilde{p}\right>.\label{ineq:2}
\end{equation}
Therefore, a simple application of Cauchy-Schwarz implies
\begin{align*}
\left<\mathcal{L}[u]\phi,\phi\right>&=-\lambda^2A^2+\left<\mathcal{L}[u]\widetilde{p},\widetilde{p}\right>\\
&>-\lambda^2A^2+\frac{\left(\lambda^2\alpha A\right)^2}{\lambda^2\alpha^2}\\
&=0
\end{align*}
as claimed.

In the case where $T_E=0$, there exists a $\psi\in X$ linearly independent from $u_x$ such that $\mathcal{L}[u]\psi=0$,
and thus $\phi_0$ and $\phi$ admit representations
\begin{align*}
\phi_0&=\alpha\chi+\gamma_1\psi+\beta u_x+p\\
\phi&=A\chi+\widetilde{p}
\end{align*}
for some constant $\gamma_1$.  It is easily verified that equations \eqref{ineq:1} and \eqref{ineq:2} still hold, and
hence
\[
\left<\mathcal{L}[u]\phi,\phi\right>=\left<-\lambda^2A\chi+\mathcal{L}[u]\widetilde{p},A\chi+\widetilde{p}\right>>0
\]
as before.
\end{proof}

Our strategy in proving Proposition \ref{coercive} is to find a particular set of translates of a
given $\phi\in\mathcal{U}_\varepsilon$ for which the inequality holds.  To this end, we
find a set of translates for each $\phi\in\mathcal{U}_\varepsilon$ such that
this set is orthogonal to the periodic null space of the second variation $\mathcal{L}[u]$
of the augmented energy functional.  This is the content of the following lemma.  

\begin{lem}\label{factoraction}
There exists an $\varepsilon>0$ and a unique $C^1$ map $\alpha:\mathcal{U}_\varepsilon\to\RM$ such that
for all $\phi\in\mathcal{U}_\varepsilon$, the function $\phi\left(\cdot+\alpha(\phi)\right)$ is orthogonal
to $u_x$.
\end{lem}

The proof is presented in \cite{BSS}, and is an easy result of the implicit function theorem.  Indeed,
if we define the functional $\eta:X\times\RM\to\RM$ by
\[
\eta(\phi,\alpha)=\int_0^T\phi(x+\alpha)u_x(x)dx,
\]
then $\frac{\partial}{\partial\alpha}\eta(\phi,\alpha)\big{|}_{(\phi,\alpha)=(u,0)}=\int_0^Tu_x^2dx>0$ and hence
the lemma follows by the implicit function theorem and the fact that by translation invariance
the function $\alpha$ can be uniquely extended to $\mathcal{U}_\varepsilon$ for $\varepsilon>0$
sufficiently small.

We now complete the proof of Proposition \ref{coercive} by proving the following lemma.

\begin{lem}\label{coercive2}
If each of the quantities $T_E$, $\{T,M\}_{a,E}$, and $\{T,M,P\}_{a,E,c}$ are positive,
there exists positive constants $\widetilde{C}$ and $\varepsilon$ such that
\[
\mathcal{E}_0(\phi)-\mathcal{E}_0(u)\geq\widetilde{C}\|\phi(\cdot+\alpha(\phi))-u\|_X^2
\]
for all $\phi\in\mathcal{U}_\varepsilon\cap\Sigma_0$.
\end{lem}

\begin{proof}
Let $\varepsilon>0$ be small enough such that Lemma \ref{factoraction} holds.
Fix $\phi\in\mathcal{U}_\varepsilon\cap\Sigma_0$ and write
\[
\phi(\cdot+\alpha(\phi))=(1+\gamma)u+\left(\beta-\frac{\gamma\left<u\right>}{T}\right)+y
\]
where $y\in\mathcal{T}_0$.  Moreover, define $v=\phi(\cdot+\alpha(\phi))-u$ and note that by
replacing $u$ with $R_\xi u$ if necessary we can assume that $\|v\|_{X}<\varepsilon$.
By Taylors theorem, we have
\[
M(a_0,E_0,c_0)=\mathcal{M}(\phi)=M(a_0,E_0,c_0)+\left<1,v\right>+\mathcal{O}\left(\|v\|_X^2\right).
\]
Since $\left<1,v\right>=\beta T$, it follows that $\beta=\mathcal{O}\left(\|v\|_{X}^2\right)$.
Similarly, we have
\[
P(a_0,E_0,c_0)=P(a_0,E_0,c_0)+\left<u,v\right>+\mathcal{O}\left(\|v\|_{X}^2\right).
\]
Moreover, a direct calculation yields
\[
\left<u,v\right>=\gamma\left(\|u\|_{L^2([0,T])}^2-\frac{\left<u\right>^2}{T}\right)+\beta\left<u\right>,
\]
where $\left<u\right>=\int_0^Tu(x)dx$.  Since $\left<u\right>^2<T\|u\|_{L^2([0,T])}^2$ by Jensen's inequality,
it follows that $\gamma=\mathcal{O}\left(\|v\|_{X}^2\right)$.

Now, by Taylor's theorem and the translation invariance of $\mathcal{E}_0$, we have
\begin{align*}
\mathcal{E}_0(\phi)&=\mathcal{E}_0\left(\phi(\cdot+\alpha(\phi))\right)\\
&=\mathcal{E}_0(u)+\left<\mathcal{E}_0'(u),v\right>
          +\frac{1}{2}\left<\mathcal{E}_0''(u)v,v\right>
          +o\left(\|v\|_{X}^2\right)\\
&=\mathcal{E}_0(u)+\frac{1}{2}\left<\mathcal{L}[u]v,v\right>
          +o\left(\|v\|_{X}^2\right).
\end{align*}
Hence, by the previous estimates on $\gamma$ and $\beta$, it follows that
\begin{align*}
\mathcal{E}_0(\phi)-\mathcal{E}_0(u)&=\frac{1}{2}\left<\mathcal{L}[u]v,v\right>
       +o\left(\|v\|_X^2\right)\\
&=\frac{1}{2}\left<\mathcal{L}y,y\right>+o\left(\|v\|_{X}^2\right).
\end{align*}
Since $y\in\mathcal{T}_0$ and $\left<y,u_x\right>=0$ by Lemma \ref{factoraction}, it follows
from Lemma \ref{positivedefinite} that
\[
\mathcal{E}_0(\phi)-\mathcal{E}_0(u)\geq\frac{C_1}{2}\|y\|^2+o\left(\|v\|_{X}^2\right).
\]
Finally, the estimates
\begin{align*}
\|y\|_X&=\left\|v-\gamma u-\beta-\frac{\gamma\left<u\right>}{T}\right\|_X\\
&\geq\left|\|v\|_X-\left\|\gamma u-\beta-\frac{\gamma\left<u\right>}{T}\right\|_X\right|\\
&\geq\left\|v\right\|_{X}-\mathcal{O}\left(\left\|v\right\|_{X}^2\right).
\end{align*}
prove that $\mathcal{E}_0(\phi)-\mathcal{E}_0(u)\geq\frac{C_1}{4}\|v\|_X^2$
for $\|v\|_X$ sufficiently small, and hence completes the proof.
\end{proof}

Proposition \ref{coercive} now clearly follows by Lemma \ref{coercive2} and the definition
of the semi-distance $\rho$.  It is now straightforward to complete the proof of Proposition \ref{os}.

\begin{proof}[Proof of Proposition \ref{os}:]
We now deviate from the methods of \cite{BSS}, \cite{GSSI} and \cite{GSSII}, and rather follow the direct method of \cite{GH}.
Let $\delta>0$ be such that Proposition \ref{coercive} holds, and let $\varepsilon\in(0,\delta)$.
Assume $\phi_0\in X$ satisfies $\rho(\phi_0,u)\leq \varepsilon$ for some small $\varepsilon>0$.
By replacing $\phi_0$ with $R_{\xi}\phi_0$ if needed, we may assume that $\|\phi_0-u\|_{X}\leq\varepsilon$.
Since $u$ is a critical point of the functional $\mathcal{E}_0$, it is clear that we have
\[
\mathcal{E}_0(\phi_0)-\mathcal{E}_0(u)\leq C_1\varepsilon^2
\]
for some positive constant $C_1$.  Now, notice that if $\phi_0\in\Sigma_0$, then the unique solution $\phi(\cdot,t)$
of \eqref{gkdv} with initial data $\phi_0$ satisfies $\phi(\cdot,t)\in\Sigma_0$ for all $t>0$.  Thus,
Proposition \ref{coercive} implies there exists a $C_2>0$ such that $\rho(\phi(\cdot,t),u)\leq C_2\varepsilon$
for all $t>0$.
Thus $\phi(\cdot,t)\in\mathcal{U}_\varepsilon$ for all $t>0$, which proves Proposition \ref{os} in this case.

If $\phi_0\notin\Sigma_0$, then we claim we can vary the constants $(a,E,c)$ slightly in order to effectively reduce
this case to the previous one.  Indeed, notice that since we have assumed $\{T,M,P\}_{a,E,c}\neq 0$
at $(a_0,E_0,c_0)$, it follows that the map
\[
(a,E,c)\mapsto\left(T(u(\;\cdot\;;a,E,c)),M(u(\;\cdot\;;a,E,c)),P(u(\;\cdot\;;a,E,c))\right)
\]
is a diffeomorphism from a neighborhood of $(a_0,E_0,c_0)$ onto a neighborhood of $(T,M(a_0,E_0,c_0),P(a_0,E_0,c_0))$.
In particular, we can find constants $a$, $E$, and $c$ with $|a|+|E|+|c|=\mathcal{O}(\varepsilon)$ such that the function
\[
\widetilde{u}=\widetilde{u}(\;\cdot\;;a_0+a,E_0+E,c_0+c)
\]
solves \eqref{gkdv}, belongs to the space $X$, and satisfies
\begin{align*}
M(a_0+a,E_0+E,c_0+a)&=\mathcal{M}(\phi_0)\\
P(a_0+a,E_0+E,c_0+c)&=\mathcal{P}(\phi_0).
\end{align*}
Defining a new augmented energy functional on $X$ by
\[
\widetilde{\mathcal{E}}(\phi)=\mathcal{E}_0(\phi)+c\mathcal{P}(\phi)+a\mathcal{M}(\phi)+ET,
\]
it follows as before that
\begin{align*}
\widetilde{\mathcal{E}}(\phi(\cdot,t))-\widetilde{\mathcal{E}}(\widetilde{u})\geq C_3\rho(\phi(\cdot,t),\widetilde{u})^2
\end{align*}
for some $C_3>0$ as long as $\rho(\phi(\cdot,t),\widetilde{u})$ is is sufficiently small.  Since $\widetilde{u}$ is a critical
point of the functional $\widetilde{\mathcal{E}}$ we have
\[
C_3\rho(\phi(\cdot,t),\widetilde{u})^2\leq\widetilde{\mathcal{E}}(\phi_0)-\widetilde{\mathcal{E}}(\widetilde{u})
             \leq C_4\|\phi_0-\widetilde{u}\|_{X}^2
\]
for some $C_4>0$.  Moreover, it follows by the triangle inequality that
\[
\|\phi_0-\widetilde{u}\|_{X}\leq\|\phi_0-u\|_{X}+\|u-\widetilde{u}\|_{X}\leq C_5\varepsilon
\]
for some $C_5>0$ and hence there is a $C_6>0$ such that
\[
\rho(\phi(\cdot,t),u)\leq\rho(\phi(\cdot,t),\widetilde{u})+\|\widetilde{u}-u\|_{X}\leq C_6\varepsilon
\]
for all $t>0$.
The proof of Proposition \ref{os}, and hence Theorem \ref{os:state}, is now complete.
\end{proof}

We would like to point out an interesting artifact of the above proof.
Notice that the step at which the sign of the quantities $\{T,M,P\}_{a,E,c}$
and $\{T,M\}_{a,E}$ came into play was in the proof of Lemma \ref{positivedefinite}, from which
we have the following corollary.

\begin{corr}
On the set $\Omega$, the quantity $\{T,M\}_{a,E}$ is positive whenever $\{T,M,P\}_{a,E,c}$ is negative and $T_E$
is positive.
\end{corr}

\begin{proof}
This is an easy consequence of Theorem \ref{orientation} and equation \eqref{osindex}.  Indeed, if
$\{T,M\}_{a,E}$ and $\{T,M,P\}_{a,E,c}$ were both negative for some $(a_0,E_0,c_0)\in\Omega$,
then by the proof of Lemma \ref{positivedefinite} we could conclude
that $\left<\mathcal{L}[u(\;\cdot\;;a_0,E_0,c_0)]\phi,\phi\right>>0$ for all $\phi\in\mathcal{T}_0$ which are orthogonal
to $u_x$.  Since this is the only time in which the signs of these quantities arise, it follows that Proposition \ref{os}
would hold thus contradicting Theorem \eqref{orientation}.
\end{proof}

It follows that we have a geometric theory of the orbital stability of periodic traveling wave solutions
of \eqref{gkdv} to perturbations of the same period as the underlying periodic wave.  In the next two
sections, we consider specific examples and limiting cases where the signs of these quantities
can be calculated.  First, we consider periodic traveling wave solutions sufficiently close to an
equilibrium solution (a local minimum of the effective potential) or to the bounding
homoclinic orbit (the separatrix solution).  By considering power-law nonlinearities in each of these
cases, we give necessary and sufficient\footnote{Except in the exceptional case of being near the homoclinic orbit for $p=4$.}
 conditions for the orbital stability of such solutions.
Secondly, we consider the KdV and prove that all periodic traveling wave solutions are orbitally
stable to perturbations of the same periodic as the underlying periodic wave.

\section{Analysis Near Homoclinic and Equilibrium Solutions}

In this section, we use the theory from section $4$ in order to prove general results about
the stability of periodic traveling wave solutions of \eqref{gkdv} in two distinguished limits:
as one approaches the solitary wave, i.e. $(a,E,c)\in\Omega$ and consider the limit $T(a,E,c)\to\infty$
for fixed $a,c$, as well as in a neighborhood of the equilibrium solution, i.e. near a non-degenerate
local minimum of the effective potential $V(u;a,c)$.  Throughout this section,
we will consider only power-law nonlinearities.

We begin with considering stability near the solitary wave.  Our main result in this limit is that the quantities
$T_E$ and $\{T,M\}_{a,E}$ are positive for $(a_0,E_0,c_0)\in\Omega$ with sufficiently large period.
Hence, the orbital stability of such a solution in this limit is determined completely by
the periodic spectral stability index $\{T,M,P\}_{a,E,c}$, which in
turn is controlled by the sign of the solitary wave stability index \eqref{solitaryindex}.
This is the content of the following theorem.

\begin{thm}
In the case of a power-law nonlinearity, i.e. $f(u)=u^{p+1}$ with $p\geq 1$, a periodic traveling wave
solution of \eqref{gkdv} of sufficiently large period and $(a,E,c)\in\Omega$
is orbitally stable if $p<4$ and exponentially unstable to perturbations
of the same period as the underlying wave if $p>4$.
\end{thm}

\begin{proof}
By the work in \cite{BrJ}, the quantity $M_a$ is negative for such $(a,E,c)\in\Omega$.
Moreover, since we are working with a power-law nonlinearity, the periodic traveling wave
solutions satisfy the scaling relation
\[
u(x;a,E,c)=c^{1/p}u\left(c^{1/2}x;\frac{a}{c^{1+1/p}},\frac{E}{c^{1+2/p}},1\right)
\]
from which we get the asymptotic relation
\[
\{T,M,P\}_{a,E,c}\sim-T_EM_a\left(\frac{2}{pc}-\frac{1}{2c}\right)P
\]
as $\Omega\ni(a,E,c)\to(0,0,c)$ for a fixed wave speed.
Since $\{T,M\}_{a,E}=M_E^2-T_EM_a$, it follows from Lemma \ref{periodLem} and
Theorem \ref{orientation} that the solutions $u(x;a,E,c)$ with $(a,E,c)\in\Omega$
of sufficiently large period are orbitally stable if $p<4$ and exponentially
unstable to periodic perturbations if $p>4$.  See \cite{BrJ} for more details of
these asymptotic calculations.
\end{proof}

Next, we consider periodic traveling wave solutions near the equilibrium solution.
We will use the methods of this paper to prove that such solutions
are orbitally stable to periodic perturbations, provided
that $a$ is sufficiently small.  To begin, we fix a wave speed $c>0$ and consider
\eqref{gkdv} with a power-law nonlinearity $f(u)=u^{p+1}$ with $p\geq 1$.
Recall that $T_E>0$ by Lemma \ref{periodLem},
and hence it suffices to prove that $\{T,M\}_{a,E}$ and $\{T,M,P\}_{a,E,c}$ are both
positive near the equilibrium solution.  By continuity, it is enough to evaluate both
these indices at the equilibrium and to show they are both positive there.  This is
the content of the following lemma.  

\begin{lem}
Consider \eqref{gkdv} with a power-law nonlinearity $f(u)=u^{p+1}$ for $p\geq 1$.
Then the quantity $M_a$ is negative for all $(a_0,E_0,c_0)\in\Omega$
such that $|a|$ is sufficiently small and the corresponding solution $u(\;\cdot\;;a_0,E_0,c_0)$
is sufficiently close to the equilibrium solution\footnote{In the case of the KdV ($p=1$), $M_a$ is negative
in a deleted neighborhood of the equilibrium solution.}.
\end{lem}

\begin{proof}
First, denote the equilibrium solution as $u_{a,c}$ and let $E^*(a,c)=V(u_{a,c};a,c)$.
It follows that
\[
\lim_{E\searrow E^*}T(a,E,c)=\frac{2\pi}{\sqrt{cp}}
\]
and that the equilibrium solution admits the expansion
\[
u_{a,c}=c^{1/p}\left(1+\frac{a}{p}\right)+\mathcal{O}(a^2).
\]
Now, solutions near the equilibrium $u_{a,c}$ can each be written as
$u(x;a,E,c)=P_{a,E,c}(k_{a,E,c}x)$, where  $k_{a,E,c}T(a,E,c)=2\pi$
and $P_{a,E,c}$ is a $2\pi$ periodic solution of the ordinary differential equation
\begin{equation*}
k_{a,E,c}^2v''+v^{p+1}-c^{1+1/p}a=0
\end{equation*}
such that
\[
P_{a,E^*,c}=u_{a,c},\;\;k_{a,E^*,c}^2=(p+1)u_{a,c}^p-1.
\]
Straightforward computations give the expansions
\begin{align*}
P_{a,E,c}(z)&=u_{a,c}+\mathcal{O}(\sqrt{E-E^*}\left(1+a^2\right)),\\
k^2_{a,E,c}&=cp+(p+1)ca+\mathcal{O}((E-E^*)+a^2).
\end{align*}
Thus, the mass $M(a,E,c)$ can be expanded as
\begin{align*}
M(a,E,c)&=\int_0^{2\pi/k_{a,E,c}}P_{a,E,c}(k_{a,E,c}z)dz\\
&=\frac{1}{k_{a,E,c}}\int_0^{2\pi}\widetilde{P}_{a,E,c}(z)dz\\
&=\frac{2\pi}{\sqrt{cp}}\left(1+\frac{(1-p)a}{2p}\right)+\mathcal{O}(\sqrt{E-E^*}+a^2)
\end{align*}
It follows that
\[
\frac{\partial}{\partial a}M(a,E,c)\big{|}_{(0,E^*,c)}=\frac{\pi(1-p)}{p\sqrt{cp}}
\]
which is negative for $p>1$.

The case $p=1$, which corresponds to the KdV equation, will be discussed in the next section.
There we will show that although $M_a$ vanishes {\it at} the equilibrium solution,
it is indeed negative for {\it nearby} periodic traveling waves with the same wave speed $c$,
i.e.
\[
\frac{\partial^2}{\partial E\partial a}M(a,E,c)\big{|}_{(0,E^*,c)}<0.
\]
\end{proof}


Next, we must determine the sign of the periodic spectral stability index $\{T,M,P\}_{a,E,c}$.
Although it follows from Theorem 4.4 in \cite{HK} that this index must be positive
\footnote{They prove the spectrum of the linearized operator $\partial_x\mathcal{L}[u]$
intersects the real axis only at $\mu=0$ for such small amplitude solutions.}, we present
an independent proof based on the periodic Evans function methods of \cite{BrJ}.
To this end, we point out that the Hamiltonian structure of the linearized operator
$\partial_x\mathcal{L}[u]$ we have the identity
\[
\{T,M,P\}_{a,E,c}=-\frac{2}{3}\tr\left(\MM_{\mu\mu\mu}(0)\right),
\]
where $\MM(\mu)$ is the corresponding monodromy operator (see Theorem 3 of \cite{BrJ} for details).
Thus, it is sufficient to show that $\tr(\MM_{\mu\mu\mu}(0))$
is negative near the equilibrium solution.  This is the content of the next lemma.

\begin{lem}\label{trppp}
Consider \eqref{gkdv} with a power-law nonlinearity $f(u)=u^{p+1}$ with $p\geq 1$.
Suppose $u_0$ is a non-degenerate local minima of the corresponding effective potential
$V(u;a,c)$.  Then $\tr\left(\MM_{\mu\mu\mu}(0)\right)<0$ at $u_0$.
\end{lem}

\begin{proof}
The key point is that the matrix $\HM(x,\mu)$ in \eqref{system} reduces to the constant
coefficient matrix
\[
\HM(\mu)=\left(
           \begin{array}{ccc}
             0 & 1 & 0 \\
             0 & 0 & 1 \\
             -\mu & -V''(u_0;a,c) & 0 \\
           \end{array}
         \right)
\]
at the equilibrium solution $u_0$.  Thus, the corresponding monodromy operator at $u_0$
can be expressed as $\MM(\mu)=\exp\left(\HM(\mu)T_0\right)$, where $T_0=\frac{2\pi}{\sqrt{p}}$.
Thus, in order to calculate the function $\tr(\MM(\mu))$, it is sufficient to calculate
the eigenvalues of the constant matrix $\HM(\mu)$.

Now, the periodic Evans function corresponding to the constant coefficient system
induced by $\HM(\mu)$ can be written as
\[
D_0(\mu,\lambda)=\det\left(\HM(\mu)-\lambda{\bf I}\right)=-\lambda^3-V''(u_0;a,c)\lambda-\mu.
\]
In particular, notice that $\frac{\partial}{\partial \lambda}D_0(\mu,\lambda)=-\lambda^2-V''(u_0;a,c)$.
Since $V''(u_0;a,c)>0$ it follows that the function $D_0(\mu,\;\cdot\;)$ will have precisely
one real root for each $\mu\in\RM$.  This distinguished root is given by the formula
\[
\gamma_1(\mu)=\underbrace{\frac{\left(\frac{2}{3}\right)^{1/3}V''(u_0)}{\left(9\mu+\sqrt{3}\sqrt{27\mu^2+4V''(u_0)^3}\right)^{1/3}}}_
    {=:\alpha(\mu)}
+\underbrace{\left(-\frac{\left(9\mu+\sqrt{3}\sqrt{27\mu^2+4V''(u_0)^3}\right)^{1/3}}{2^{1/3}3^{2/3}}\right)}_{=:\beta(\mu)}.
\]
Defining $\omega=\exp(2\pi i/3)$ to be the principle third root of unity, the two complex eigenvalues of $\HM(\mu)$
can be written as $\gamma_2(\mu)=\omega\alpha(\mu)+\overline{\omega}\beta(\mu)$ and
$\gamma_3(\mu)=\overline{\omega}\alpha(\mu)+\omega\beta(\mu)$, and hence
\[
\tr\left(\MM(\mu)\right)=\exp\left(\gamma_1(\mu)T_0\right)+\exp\left(\gamma_2(\mu)T_0\right)
      +\exp\left(\gamma_3(\mu)T_0\right).
\]
Now, a straightforward, yet tedious, calculation using the facts that $1+\omega+\overline{\omega}=0$ and $\omega^2=\overline{\omega}$
implies that
\[
\tr\left(\MM_{\mu\mu\mu}(0)\right)=9T_0^2\left(\alpha''(0)\beta'(0)+\alpha'(0)\beta''(0)\right)
               +3T_0^3\left(\alpha'(0)^3+\beta'(0)^3\right).
\]
Moreover, from the definitions of $\alpha$ and $\beta$ we have
\[
\alpha'(0)=-\frac{1}{2V''(u_0)}=\beta'(0),\;\;{\rm and }\;\;
     \alpha''(0)=\frac{\sqrt{3}}{4V''(u_0)^{5/2}}=-\beta''(0).
\]
Therefore, we have the equality
\[
\tr\left((\MM_{\mu\mu\mu}(0)\right)=-\frac{6\pi^3}{V''(u_0)^{9/2}},
\]
which is clearly negative.
\end{proof}

\begin{rem}
Notice that the proof of Lemma \ref{trppp} holds verbatim if one considers more general
nonlinearities.  In particular, one could obtain a similar expression for $\tr(\MM_{\mu\mu}(0))$
and show the modulational instability index from \cite{BrJ} always vanishes at the equilibrium
solutions of the traveling wave ordinary differential equation \ref{travelode}.  Moreover, this
proof holds regardless of the size of the parameter $a$.
\end{rem}

Therefore, it follows that in the case of a power-nonlinearity and solutions sufficiently close
to a non-degenerate minima of the effective potential, each of the quantities
$T_E$, $\{T,M\}_{a,E}$, and $\{T,M,P\}_{a,E,c}$ are all positive.  Therefore, Theorem \ref{os:state}
immediately yields the following result.

\begin{thm}
Consider equation \eqref{gkdv} with a power-law nonlinearity $f(u)=u^{p+1}$ for $p\geq 1$.
Then the periodic traveling wave solutions $u(x;a,E,c)$ with $(a,E,c)\in\Omega$ and
$a^2+(E-E^*)^2$ sufficiently small are orbitally stable in the sense of Theorem \ref{os:state}.
\end{thm}

\section{The Korteweg-de Vries Equation}

In this section, we will apply the general theory from section 4 in order to prove that periodic traveling wave solutions
of \eqref{gkdv} with $f(u)=u^2$ and $c>0$ are orbitally stable with respect to periodic perturbations if and only
if they are spectrally stable to such perturbations.  To this end, we notice that solutions of the KdV equation
\begin{equation}
u_t=u_{xxx}+(u^2)_x-cu_x\label{kdv}
\end{equation}
are invariant under the scaling transformation
\[
u(x,t)\mapsto c\;u(\sqrt{c}x,c^{3/2}t),
\]
and hence, by \eqref{quad2}, the stationary periodic traveling wave solutions satisfy the identity
\[
u(x;a,E,c)=c\;u\left(c^{1/2}x;\frac{a}{c},\frac{E}{c^3},1\right).
\]
Thus, by scaling we may always assume that $c=1$ in \eqref{kdv}.  Moreover, we may always
assume that $a=0$ due to the Galilean invariance of the KdV.  Therefore, it is
sufficient to determine the stability of periodic traveling wave solutions of \eqref{kdv}
of the form $u(x;0,E,1)$.  In order to do so, we need the following easily proved lemma:

\begin{lem}\label{posint}
Let $\mu$ be a (Borel) probability measure on some interval $I\subset\RM$, and let $f,g:I\to\RM$
be bounded and measurable functions.  Then
\begin{equation}
\int_I f(x)g(x)d\mu-\left(\int_I f(x)d\mu\right)\left(\int_I g(x)d\mu\right)
    =\frac{1}{2}\int_{I\times I}\left(f(x)-f(y)\right)\left(g(x)-g(y)\right)d\mu_x d\mu_y.\label{inteqn}
\end{equation}
In particular, if both $f$ and $g$ are strictly increasing or strictly decreasing, and if the support of $\mu$
is not reduced to a single point, then
\[
\int_I f(x)g(x)d\mu > \left(\int_I f(x)d\mu\right)\left(\int_I g(x)d\mu\right).
\]
\end{lem}

The proof of this lemma is a trivial result of Fubini's theorem, as one can see by writing the left hand
side of \eqref{inteqn} as an iterated integral and simplifying the resulting expression.  Now,
recall from Lemma \ref{periodLem} that $T_E>0$ for periodic traveling wave solutions of \eqref{kdv}.
To conclude orbital stability, we must identify the signs of the Jacobians $\{T,M\}_{a,E}$
and $\{T,M,P\}_{a,E,c}$.  The main technical result we need for this section
is the following lemma, which uses Lemma \ref{posint} to guarantee the sign
of the quantity in \eqref{osindex} is completely determined by the Jacobian $\{T,M,P\}_{a,E,c}$.

\begin{lem}\label{kdv:lem}
If $f(u)=u^2$ in \eqref{gkdv}, then $\{T,M\}_{a,E}>0$ for all $(a_0,E_0,c_0)\in\Omega$ which do not correspond
to the unique equilibrium solution.
\end{lem}

\begin{proof}
First, notice that since $\nabla_{a,E,c}K(a,E,c)=\left(M,T,P\right)$, it follows that
$\{T,M\}_{a,E}=M_E^2-T_EM_a$, and hence by Lemma \ref{periodLem} it is enough
to prove that $M_a<0$.  Moreover, by our above remarks it is enough to consider the case $c=1$
and $a=0$.  It follows for $f$ given as above, we can find functions $u_1$, $u_2$, $u_3$
which depend smoothly on $(a,E,c)$ within the domain $\Omega$ such that
\[
3\left(E-V(u;0,1)\right)=(u-u_1)(u-u_2)(u_3-u).
\]
Notice that the assumption that we are not at the equilibrium solution implies
that the roots $u_i$ are distinct, and moreover that $V'(u_i;0,1)\neq 0$.
Since $E-V(u_{i};0,1)=0$ on $\Omega$, it follows that
\[
\frac{\partial u_{i}}{\partial a}=\frac{u_i}{V'(u_i;0,1)}.
\]
Since $u_1<0$ and $u_2,\;u_3>0$, we have
\begin{equation}
\frac{\partial u_1}{\partial_a}<0,\;\;\frac{\partial u_2}{\partial_a}<0,\;\;{\rm and}\;\;\frac{\partial u_3}{\partial_a}>0.\label{a-zeroes}
\end{equation}
Moreover, since $u_1+u_2+u_3=\frac{3c}{2}$ we have the relation
\begin{equation}
\frac{\partial u_2}{\partial a}+\frac{\partial u_3}{\partial a}=-\frac{\partial u_1}{\partial a}>0\label{a-der-sum}
\end{equation}
on $\Omega$.

Now, by making the change of variables $u\mapsto s(\theta)=u_2\cos^2(\theta)+u_2\sin^2(\theta)$, we have
$du=2\sqrt{(u-u_2)(u_3-u)}d\theta$ and hence we may express the mass of $u(x;a,E,c)$ as
\begin{align}
M(a,E,c)&=\sqrt{2}\int_{u_2}^{u_3}\frac{u\;du}{\sqrt{E-V(u;a,c)}}\nonumber\\
&=2\sqrt{6}\int_0^{\pi/2}\frac{s(\theta)\;d\theta}{\sqrt{s(\theta)-u_1}}\label{mass:var}.
\end{align}
Notice we suppress the dependence of $s(\theta)$ on the parameters $(a,E,c)$.
Defining $\sigma(\theta)=\sqrt{s(\theta)-u_1}$, a straightforward computation using \eqref{a-der-sum} shows that
the derivative of the integrand in \eqref{mass:var} with respect to the parameter $a$ can be expressed as
\begin{align*}
\frac{\partial}{\partial a}\left(\frac{s(\theta)}{\sqrt{s(\theta)-u_{1}}}\right)&=
   \frac{\partial u_2}{\partial a}\left(\frac{\cos^2(\theta)}{2\sigma(\theta)}\right)
        +\frac{\partial u_3}{\partial a}\left(\frac{\sin^2(\theta)}{2\sigma(\theta)}\right)\\
&-\left(\frac{\partial u_2}{\partial a}+\frac{\partial u_3}{\partial a}\right)
        \left(\frac{s(\theta)}{2\sigma(\theta)^3}\right)-\frac{u_1}{2\sigma(\theta)^3}
        \left(\frac{\partial u_2}{\partial a}\cos^2(\theta)+\frac{\partial u_3}{\partial a}\sin^2(\theta)\right)\\
&=\frac{\partial u_2}{\partial a}\left(\frac{\cos^2(\theta)-\sin^2(\theta)}{2\sigma(\theta)}\right)
      +\left(\frac{\partial u_2}{\partial a}+\frac{\partial u_3}{\partial a}\right)
      \left(\frac{\sin^2(\theta)}{2\sigma(\theta)^3}\right)
      -\left(\frac{\partial u_2}{\partial a}+\frac{\partial u_3}{\partial a}\right)
      \frac{s(\theta)}{2\sigma(\theta)^3}\\
&-u_{1}\left(\frac{\partial u_2}{\partial a}\left(\frac{\cos^2(\theta)-\sin^2(\theta)}{2\sigma(\theta)^3}\right)
              +\left(\frac{\partial u_2}{\partial a}+\frac{\partial u_3}{\partial a}\right)
             \left(\frac{\sin^2(\theta)}{2\sigma(\theta)^3}\right)\right).
\end{align*}
With a little more algebra, this may be rewritten as
\begin{align*}
\frac{\partial}{\partial a}\left(\frac{s(\theta)}{\sqrt{s(\theta)-u_{1}}}\right)&=
\frac{\partial u_2}{\partial a}\left(\frac{\cos^2(\theta)-\sin^2(\theta)}{2\sigma(\theta)}\right)
     -u_1\frac{\partial u_2}{\partial a}\left(\frac{\cos^2(\theta)-\sin^2(\theta)}{2\sigma(\theta)^3}\right)\\
&-\left(\frac{\partial u_2}{\partial a}+\frac{\partial u_3}{\partial a}\right)
       \left(\frac{s(\theta)\cos^2(\theta)-u_1\left(\cos^2(\theta)-\sin^2(\theta)\right)}{2\sigma(\theta)^3}\right).
\end{align*}
Since the functions $\cos^2(\theta)-\sin^2(\theta)$ and $\sigma(\theta)^{-1}$ are strictly decreasing on the
interval $(0,\pi/2)$, it follows from Lemma \ref{posint} that
\[
\int_0^{\pi/2}\frac{\cos^2(\theta)-\sin^2(\theta)}{\sigma^m(\theta)}\;d\theta> 0
\]
for any $m>0$.  Evaluating the above expression at $(a,E,c)=(0,E,1)\in\Omega$ implies that $s(\theta)>0$
for all $\theta\in(0,\pi/2)$, and hence \eqref{a-zeroes} and \eqref{a-der-sum} imply that
\[
\int_0^{\pi/2}\frac{\partial}{\partial a}\left(\frac{s(\theta)}{\sqrt{s(\theta)-u_{1}}}\right)d\theta < 0
\]
at $(0,E,1)$, from which the lemma follows.
\end{proof}

Therefore, our main theorem on  the stability of periodic traveling wave solutions of the Korteweg-de Vries
equation follows by Theorem \ref{os:state}, Theorem \ref{orientation}, and Lemma \ref{kdv:lem}.

\begin{thm}\label{thm:kdv}
Let $u(x;a_0,E_0,c_0)$ be a periodic traveling wave solution of \eqref{kdv}, i.e. let $(a_0,E_0,c_0)\in\Omega$.
Then $u$ is orbitally stable in the sense of Theorem \ref{os:state}
if and only if the solution is spectrally stable to perturbations of the same period,
i.e. if and only if $\{T,M,P\}_{a,E,c}>0$ at $(a_0,E_0,c_0)$.
\end{thm}



An interesting corollary of Theorem \ref{thm:kdv} applies to cnoidal wave solutions of the KdV.
It was suggested by Benjamin \cite{Benjamin} that such solutions should be
stable to perturbations of the same period.  This conjecture has indeed been proved both by
using the complete integrability of the KdV \cite{McKean} \cite{BD}
and by variational methods as in the present paper \cite{PBS}.
In particular, in \cite{BD} it was shown that the cnoidal solutions of the KdV are
spectrally stable to localized perturbations, and are linearly stable to perturbations with
the same period as the underlying wave.  Clearly then such solutions are spectrally stable
with respect to periodic perturbations.  Paired with Theorem \ref{thm:kdv}, this
provides another verification Benjamin's conjecture in the case where the cnoidal
wave has positive wave speed.

\begin{corr}
The cnoidal wave solutions of \eqref{gkdv} with $f(u)=u^2$ of the form
\[
u(x,t)=u_0+12k^2\kappa^2\cn^2\left(\kappa\left(x-x_0-\left(8k^2\kappa^2-4\kappa^2+u_0\right)t\right)\right),
\]
with $k\in[0,1)$ and $\kappa,\;x_0$, and $u_0$ real constants, are orbitally stable in the sense of Theorem \ref{os:state}
if the wave speed $8k^2\kappa^2-4\kappa^2+u_0$ is positive.
\end{corr}

\section{Concluding Remarks}

In this paper, we extended the recent results of \cite{BrJ} in order to determine sufficient conditions for the orbital
stability of the four-parameter family of periodic traveling wave solutions of the generalized Korteweg-de Vries
equation \eqref{gkdv}.  By extending the methods of \cite{BSS} to the periodic case,
a new geometric condition was derived in terms of the conserved
quantities of the gKdV flow restricted to the manifold of periodic traveling wave solution, and it
was shown how this could be translated to a condition on the Hessian of the classical action
of the ordinary differential equation governing the periodic traveling waves.  As a byproduct
of this theory, it was shown that such solutions of the KdV are orbitally stable to perturbations of the
same period as the underlying wave if and only if they are spectrally stable to periodic perturbations
of the same period.

There are several points in which this theory is still lacking.  First off, it is not clear what happens in the
case $T_E< 0$.  In the solitary wave theory, the existence of two negative eigenvalues
of the second variation $\mathcal{L}[u]$ indicates instability.  Also, if $T_E>0$ and $\{T,M\}_{a,E}<0$
it is not clear whether this implies orbital instability, although we conjecture
this is indeed the case.  We would like to show in the case $\{T,M\}_{a,E}\{T,M,P\}_{a,E,c}<0$ and $T_E>0$
that there exists a 1-parameter family of functions in $\Sigma_0$ which contain the solution
$u(x;a_0,E_0,c_0)$ such that the augmented energy functional $\mathcal{E}_0$ has a strict
local maximum at $(a_0,E_0,c_0)$.  However, it is not clear how to do this in a reasonable manner: mainly,
one must fix the period, mass, and momentum along this curve, and the derivative of this curve at $(a_0,E_0,c_0)$
must also be in $X$.  The existence of such an instability would stand in stark contrast to the solitary
wave case, where the orbital stability is equivalent to the spectral instability (except possibly on the transition curve).
However, it seems quite possible that such a situation arises due to the fact that the solitary waves
are a co-dimension two subset of the family of traveling wave solutions.

{\bf Acknowledgements:} The author would like to thank Jared C. Bronski for many useful
discussions, especially as concerned with the proof of Lemma \ref{periodLem},
and for looking over the preliminary work regarding this paper.

\bibliography{gkdvOrbitalStabilitybib}

\end{document}